\documentclass{article}
\usepackage[latin1]{inputenc}
\usepackage{amsmath}
\usepackage{amsthm,amssymb,amsfonts}
\usepackage{graphicx}
\usepackage{verbatim}
\usepackage{color,cite}
\usepackage{latexsym}
\usepackage{lscape}
\usepackage[T1]{fontenc}
\usepackage[all,cmtip]{xy}
\usepackage{tikz}

\theoremstyle{plain}
\newtheorem{theorem}{Theorem}[section]
\newtheorem{lemma}[theorem]{Lemma}
\newtheorem{proposition}[theorem]{Proposition}
\newtheorem{corollary}[theorem]{Corollary}

\newtheorem*{definition}{Definition}

\newtheorem*{warning}{Warning}

\newtheorem*{remark}{Remark}
\newtheorem{convention}{Convention}

\title{The Cantor-Bendixson Rank of Certain Bridgeland-Smith Stability Conditions}
\author{David Aulicino\thanks{This material is based upon work supported by the ERC Starting Grant ``Quasiperiodic'' of Artur Avila, and later by the National Science Foundation under Award No. DMS - 1204414.}}
\date{}

\begin{document}

\newcommand{\splin}{$\text{SL}_2(\mathbb{R})$}
\newcommand{\spolin}{$\text{SO}_2(\mathbb{R})$}
\newcommand{\TeichDisk}{$\Phi(\text{SL}_2(\mathbb{R})(X,\omega))$}
\newcommand{\RankOne}{$\mathcal{D}_g (1)$}
\newcommand{\nc}{\newcommand}
\nc\bB{\mathbb{B}}
\nc\bC{\mathbb{C}}
\nc\bD{\mathbb{D}}
\nc\bE{\mathbb{E}}
\nc\bF{\mathbb{F}}
\nc\bG{\mathbb{G}}
\nc\bH{\mathbb{H}}
\nc\bI{\mathbb{I}}
\nc{\bJ}{\mathbb{J}}
\nc\bK{\mathbb{K}}
\nc\bL{\mathbb{L}}
\nc\bM{\mathbb{M}}
\nc\bN{\mathbb{N}}
\nc\bO{\mathbb{O}}
\nc\bP{\mathbb{P}}
\nc\bQ{\mathbb{Q}}
\nc\bR{\mathbb{R}}
\nc\bS{\mathbb{S}}
\nc\bT{\mathbb{T}}
\nc\bU{\mathbb{U}}
\nc\bV{\mathbb{V}}
\nc\bW{\mathbb{W}}
\nc\bY{\mathbb{Y}}
\nc\bX{\mathbb{X}}
\nc\bZ{\mathbb{Z}}

\nc\cA{\mathcal{A}}
\nc\cB{\mathcal{B}}
\nc\cC{\mathcal{C}}
\nc\cD{\mathcal{D}}
\nc\cE{\mathcal{E}}
\nc\cF{\mathcal{F}}
\nc\cG{\mathcal{G}}
\nc\cH{\mathcal{H}}
\nc\cI{\mathcal{I}}
\nc{\cJ}{\mathcal{J}}
\nc\cK{\mathcal{K}}
\nc\cM{\mathcal{M}}
\nc\cN{\mathcal{N}}
\nc\cO{\mathcal{O}}
\nc\cP{\mathcal{P}}
\nc\cQ{\mathcal{Q}}
\nc\cS{\mathcal{S}}
\nc\cT{\mathcal{T}}
\nc\cU{\mathcal{U}}
\nc\cV{\mathcal{V}}
\nc\cW{\mathcal{W}}
\nc\cY{\mathcal{Y}}
\nc\cX{\mathcal{X}}
\nc\cZ{\mathcal{Z}}

\maketitle

\begin{abstract}
We provide a novel proof that the set of directions that admit a saddle connection on a meromorphic quadratic differential with at least one pole of order at least two is closed, which generalizes a result of Bridgeland and Smith, and Gaiotto, Moore, and Neitzke.  Secondly, we show that this set has finite Cantor-Bendixson rank and give a tight bound.  Finally, we present a family of surfaces realizing all possible Cantor-Bendixson ranks.  The techniques in the proof of this result exclusively concern Abelian differentials on Riemann surfaces, also known as translation surfaces.  The concept of a ``slit translation surface'' is introduced as the primary tool for studying meromorphic quadratic differentials with higher order poles.
\end{abstract}

\tableofcontents

\section{Introduction}

The work of \cite{BridgelandSmithStabConds} drew a deep connection between spaces of meromorphic quadratic differentials with poles and simple zeros and spaces of stability conditions on certain categories.  These differentials with some additional assumptions are called GMN-differentials in \cite{BridgelandSmithStabConds}.  The saddle connections of the quadratic differentials correspond to stable objects of the stability conditions.  Furthermore, wall-crossings occur along a subset of the set of directions (called phases in \cite{BridgelandSmithStabConds}) that admit a saddle connection.  The result of \cite{BridgelandSmithStabConds} provides a new justification for studying meromorphic quadratic differentials with higher order poles.

This paper considers meromorphic quadratic differentials with at least one pole of order at least two, and zeros of arbitrary finite order.  We answer a question of Ivan Smith concerning a fundamental property of the set of directions on these differentials that admit saddle connections, and provide a new proof of a lemma of \cite{BridgelandSmithStabConds, GMNWallHitchinWKB}.

There is a growing body of literature concerning meromorphic differentials including the recent works of \cite{BoissyConnCompsStrMeroDiffs, BoissyModSpMeroDiffs, GuptaMeroQuadHalfPlane, GuptaWolfQuadDiffsHalfPlaneStrcts}.  This paper falls in line with those works and answers a question inspired by stability conditions.  In addition to being of interest to those working with meromorphic differentials, we hope that the results of this paper are also of interest to those who study stability conditions.

The techniques used in this paper lie firmly in the realm of translation surfaces, which are given by Abelian differentials on Riemann surfaces.  We introduce the concept of a ``slit translation surface,'' which is an ordinary translation surface with a collection of marked line segments.  This provides the key object that makes the problem tractable.  The connection of these objects to meromorphic quadratic differentials follows from \cite{Strebel} and \cite{GuptaMeroQuadHalfPlane}.  We begin by recalling the celebrated result \cite[Thm. 2]{MasurClosedTraj}.

\begin{theorem}[\cite{MasurClosedTraj}]
\label{MasurThm}
Given a non-zero holomorphic quadratic differential on a Riemann surface, the set of directions that admit a cylinder is \emph{dense} in the circle.
\end{theorem}

The following is a trivial corollary of Theorem \ref{MasurThm}.  However, there is a far simpler proof, which we provide in the next section, that does not rely on Theorem \ref{MasurThm}.

\begin{corollary}
\label{SCDense}
Given a non-zero holomorphic quadratic differential on a Riemann surface, the set of directions that admit a saddle connection is \emph{dense} in the circle.
\end{corollary}

The assumption holomorphic can be weakened to meromorphic with at most simple poles in the previous corollary.  However, there is a striking difference in behavior when a meromorphic quadratic differential is permitted to have a pole of order two or more.  The following lemma was proven in \cite{GMNWallHitchinWKB} and \cite[Lem. 4.11]{BridgelandSmithStabConds}.

\begin{lemma}[\cite{BridgelandSmithStabConds, GMNWallHitchinWKB}]
\label{GMNClosed}
Given a meromorphic quadratic differential on a Riemann surface such that the differential has at least one pole of order at least two and all simple zeros, or more precisely, a GMN-differential, the set of directions that admit a saddle connection is \emph{closed}.
\end{lemma}

As a consequence of the work in this paper, we give a new proof of their result in complete generality.  We observe that our proof differs from theirs in that they prove that the set is closed by showing that the complement is open, and we prove that it is closed by showing that the limits of all convergent sequences lie in the set.

\begin{theorem}
\label{DoubPoleClosed}
Given a meromorphic quadratic differential with at least one pole of order at least two on a Riemann surface, the set of directions that admit a saddle connection is closed.
\end{theorem}

The proof is given in Section \ref{PfGenGMNLemSect} after all of the necessary prerequisites have been established.  The main result of this paper is

\begin{theorem}
\label{FinCBRankMainThm}
Given a meromorphic quadratic differential with at least one pole of order at least two on a Riemann surface, the set of directions that admit a saddle connection has finite Cantor-Bendixson rank.
\end{theorem}

In fact, once the proper terminology is introduced, we give a tight explicit upper bound on the Cantor-Bendixson rank in Theorem \ref{CBRankBd} and produce an infinite family of surfaces realizing the maximum possible Cantor-Bendixson rank in Section \ref{SnExSect}.  The family of examples given in Section \ref{SnExSect} demonstrate that arbitrarily large Cantor-Bendixson rank can be achieved if the genus is sufficiently large.

\

\noindent \textbf{Acknowledgments:} The author would like to thank Ivan Smith for posing the problem that inspired this paper, and for his continued interest.  The author is also extremely grateful to Howard Masur for listening to numerous presentations of earlier forms of this work, and for his patience and feedback throughout the discussions.  He is also grateful to him for his careful and thoughtful feedback on an earlier version of this paper.  He also thanks Giovanni Forni for input that lead to a simplification in the proof of Lemma \ref{InfSCImpFinCyl}.  Finally, the author thanks Jon Chaika and Anton Zorich for helpful discussions.

\section{Definitions and Preliminaries}

\noindent \textbf{Abelian and Quadratic Differentials:} Let $X$ be a finite genus Riemann surface of genus $g \geq 2$.  An Abelian differential is $1$-form that can be expressed in local coordinates as $f(z)\,dz$, where $f(z)$ is meromorphic.  A quadratic differential is a tensor that can be expressed in local coordinates as $g(z)\,dz^2$, where $g(z)$ is meromorphic.  Throughout this paper, $\omega$ will always refer to an Abelian differential, and $q$ will refer to a quadratic differential.  An Abelian (resp. quadratic) differential canonically determines a horizontal and vertical foliation on $X$ by $\{\text{Im}(\omega) = 0\}$ (resp. $\{\text{Re}(q) > 0\}$) and $\{\text{Re}(\omega) = 0\}$ (resp. $\{\text{Re}(q) < 0\}$), respectively.  Pairs $(X, \omega)$ are commonly called \emph{translation surfaces} and pairs $(X, q)$ are commonly called \emph{half-translation surfaces}.

\begin{definition}
Given $(X, q)$, let $\Sigma \subset X$ be a finite collection of marked points on $X$ such that $\Sigma$ contains all points at which $q$ has a zero or simple pole, but not a pole of order two or more.  A \emph{saddle connection} on $(X,q)$ is a finite length segment contained in a leaf of a foliation such that its end points are elements of $\Sigma$.
\end{definition}

\begin{definition}
Given a regular closed leaf of a foliation on $(X, q)$, the maximal set of leaves homotopic to it is called a \emph{cylinder}.  If the distance between the boundaries of a cylinder is finite, we call the cylinder a \emph{finite cylinder}.  If a leaf of a foliation on $(X, q)$ is not closed, its closure is called a \emph{minimal component}.  Any subset of a surface that is either a cylinder or a minimal component is called an \emph{invariant component}.
\end{definition}

Observe that invariant components come with a canonical direction that is well-defined modulo $\pi$, which we call the \emph{direction} of the invariant component.  If two invariant components have the same direction modulo $\pi$, then we say they are \emph{parallel}.

We provide an elementary proof of Corollary \ref{SCDense} here.  

\begin{proof}[Proof of Corollary \ref{SCDense}]
Let $(X, q)$ be the Riemann surface carrying the non-zero holomorphic quadratic differential.  Fix a zero $z$ of $q$.  Consider a sector of angle $\theta > 0$ defined by two rays emanating from $z$.  As the radius of the sector increases, the area defined by the sector increases.  Since $(X,q)$ has finite area, the sector must contain a zero of $q$ (not necessarily distinct from $z$) for a sufficiently large radius of the sector.  Furthermore, this is true regardless of the direction of the sector or the angle $\theta$.  Hence, there is a dense set of directions in $(X,q)$ that admit a saddle connection.
\end{proof}

\noindent \textbf{Strata of Differentials:} For $\kappa$ a partition of $2g-2$, the stratum $\cH(\kappa)$ is the moduli space of all genus $g$ translation surfaces with orders of zeros of the Abelian differential specified by $\kappa$.  For $\kappa$ a partition of $4g-4$, where we permit any finite number of simple poles of order $-1$, the stratum $\cQ(\kappa)$ is the moduli space of all genus $g$ half-translation surfaces with order of zeros and the number of simple poles of the quadratic differential specified by $\kappa$.  For meromorphic Abelian or quadratic differentials with higher order poles there is a discussion of strata of such meromorphic differentials discussed in \cite{BoissyConnCompsStrMeroDiffs, BoissyModSpMeroDiffs}.  However, for our purposes, it suffices to regard such a stratum as a set specifying the orders of all of the zeros and poles.

\

\noindent \textbf{Cantor-Bendixson Rank:} We follow the definition of Cantor-Bendixson rank given in \cite[Ch. 6C]{KechrisClassDescSetTheory}.  For the purpose of this paper it suffices to consider all sets as subsets of the unit interval in the real line.

\begin{definition}
Let $S \subset \bR$.  Define the \emph{derived set} $S^*$ of $S$ to be the subset of $S$ such that all isolated points of $S$ are removed.  Denote by $S^{*n}$ the $n$'th derived set of $S$, and $S^{*0} = S$.
\end{definition}

\begin{definition}
Let $S \subset \bR$.  The \emph{Cantor-Bendixson rank} of $S$ is defined to be the smallest non-negative integer $n$ such that $S^{*n+1} = S^{*n}$.
\end{definition}

For example, the empty set or any perfect set has Cantor-Bendixson rank zero.  Furthermore, all non-empty finite sets have Cantor-Bendixson rank one.  The following lemma is obvious.

\begin{lemma}
\label{CantorBendClosedSetBd}
Let $G \subset \bR$ be a closed set such that the only perfect subset of $G$ is the empty set.  If $F \subset G$ is closed and $G$ has Cantor-Bendixson rank $k$, then the Cantor-Bendixson rank of $F$ is bounded above by $k$.
\end{lemma}

\noindent \textbf{Saddle Connection Directions:}

\begin{definition}
Let $(X, q)$ be a Riemann surface carrying a quadratic differential.  Let $\cF_{\theta}$ denote the vertical foliation of $(X, e^{i\theta}q)$.  Define the set
$$\Theta(X,q) = \{ \theta \in [0,\pi) | \cF_{\theta} \text{ admits a saddle connection} \}.$$
The same definition holds if $q$ is replaced by an Abelian differential $\omega$ and the angle is taken modulo $\pi$.
\end{definition}

The following lemma is obvious, but we state it to illustrate that there are no complications in the presence of poles of arbitrary order.

\begin{lemma}
\label{SCAlwExistLem}
Let $(X,q)$ be a Riemann surface of genus $g \geq 1$ and a non-empty set $\Sigma$ as defined above.  Then for some $\theta$, $(X,e^{i\theta}q)$ admits a saddle connection.  Furthermore, the Cantor-Bendixson rank of $\Theta(X,q)$ is always positive because $\Theta(X, q) \not= \emptyset$ and $\Theta(X, q)$ is not perfect.
\end{lemma}

\begin{proof}
First we prove that $\Theta(X, q) \not= \emptyset$.  Let $p \in \Sigma$.  Consider discs $D_r(p)$ of radius $r > 0$ about $p$.  We consider $r$ tending to infinity.  As soon as we find a value of $r$ for which the disc $D_r(p) \subset (X, q)$ contains another element of $\Sigma$, or self-intersects, we can consider the straight trajectory through from $p$ to the other element of $\Sigma$, or to the self-intersecting boundary to get a saddle connection.  However, one of these possibilities must occur, otherwise, $(X, q)$ would be a simply connected disc with a marked point.

Finally, it is well known that for all non-zero quadratic differentials with finitely many finite order zeros and poles, $\Theta(X, q)$ is at most countable.  Hence, it can never contain a non-empty perfect subset because all non-empty perfect sets are uncountable.  Thus, successive derivations of this set can only stabilize on the empty set.
\end{proof}

A priori, it is not clear that the Cantor-Bendixson rank of $\Theta$ is always finite.  This will follow as a corollary of Theorem \ref{CBRankBd}.

\

\noindent \textbf{Canonical Double Covering:} We recall the canonical double covering of a quadratic differential.  In particular, we bring to the reader's attention that this construction, which is typically discussed in the context of meromorphic quadratic differentials with at most simple poles works for meromorphic quadratic differentials with poles of arbitrary finite order.  Let $(X, q)$ be a quadratic differential in the stratum $\cQ(k_1, \ldots, k_n, k_{n+1}, \ldots, k_{n+m})$, where the numbers are sorted so that $k_1, \ldots, k_n$ are even and $k_{n+1}, \ldots, k_{n+m}$ are odd.  We permit $n \geq 0$ and $m \geq 0$, but $n+m > 0$.  Here we only assume that $k_i \in \bZ$ and $|k_i| < \infty$ for all $i$.  Let $p_i \in X$ be the singularity of $q$ of order $k_i$, for all $i$.  Consider the double cover of $(X, q)$ branched only over the points $\{p_{n+1}, \ldots, p_{n+m}\}$, and denote the resulting surface by $(\tilde X, \omega)$.  We call $(\tilde X, \omega)$ the \emph{canonical double cover} of $(X, q)$.  Observe that $\omega$ is a global square-root so that it is in fact a meromorphic $1$-form, and $(\tilde X, \omega)$ is in the stratum $\cH(k_1/2, \ldots, k_n/2, k_1/2, \ldots, k_n/2, k_{n+1}+1, \ldots, k_{n+m}+1)$.  Again, we remark that this familiar formula for the stratum containing $(\tilde X, \omega)$ works equally well when $k_i < -1$.

\begin{lemma}
\label{CBRankCov}
Let $(X, q)$ be a Riemann surface carrying a meromorphic quadratic differential, and let $(\tilde X, \omega)$ be the resulting canonical double covering.  Then
$$\Theta(X,q) =  \Theta(\tilde X, \omega).$$
In particular, the Cantor-Bendixson rank of the two sets are equal.
\end{lemma}

\begin{proof}
We have $\Theta(X,q) \subseteq \Theta(\tilde X, \omega)$ because $\Theta(X,q) \subset [0,\pi)$, the direction of a saddle connections on $(X, q)$ is preserved up to rotation by $\pi$ on $(\tilde X, \omega)$, and saddle connections on $(X, q)$ lift to one or two parallel saddle connections above.

On the other hand, we claim that we also have $\Theta(\tilde X, \omega) \subseteq \Theta(X,q)$.  Let $\sigma \subset (\hat X, \omega)$ be a saddle connection between the (not necessarily distinct zeros) $z_1$ and $z_2$.  Let $\pi: (\tilde X, \omega) \rightarrow (X, q)$ denote the canonical double covering map.  Then $\pi(\sigma)$ is a parallel union of saddle connections from $\pi(z_1)$ to $\pi(z_2)$, so there is indeed a saddle connection on $(X, q)$ in the same direction as $\sigma$.  Hence, the directions containing saddle connections are the same.
\end{proof}

In light of this lemma, any investigation of the Cantor-Bendixson rank of the set of saddle connection directions on $(X, q)$ can be carried out for Abelian differentials via the canonical double covering.

\section{Surgeries on Quadratic Differentials}

We introduce local surgeries of neighborhoods of poles of a quadratic differential so that any Riemann surface carrying a meromorphic quadratic differential with poles of arbitrary finite order can be associated in a meaningful way to a Riemann surface possibly with boundary carrying a meromorphic quadratic differential with at most simple poles.  The boundary can be regarded as marked line segments in the flat picture.  These marked line segments are examples of the \emph{slits} defined below.  The way we perform the surgeries will depend on the order of the poles.  We divide the cases into two sections.  In all cases we call the surgery the \emph{partial pole surgery}.  Since this is the only surgery that will appear in this paper, we will simply refer to it as the surgery.  In all cases we denote the surface after surgery with hats, i.e. $(X, q)$ after the surgery becomes $(\hat X, \hat q)$.

\subsection{Surgery of a Double Pole}

By \cite{Strebel}, if $(X, q)$ is a Riemann surface carrying a quadratic differential with a double pole at $z \in X$, then there exists a complex unit $e^{i\theta}$ such that the vertical foliation of $(X, e^{i\theta}q)$ around $z$ is given by concentric circles or equivalently as a half-infinite cylinder foliated by circles, \cite[Fig. 9]{Strebel}.

Define the \emph{partial pole surgery} for double poles as follows.  For each double pole of $q$, multiply $q$ by an appropriate complex unit that realizes the flat structure as a half-infinite cylinder foliated by closed circles and bounded by a union of saddle connections.  Consider the closed trajectory in the half-infinite cylinder of unit flat distance from the boundary of the cylinder, cut the cylinder along this trajectory and discard the half-infinite cylinder to get a translation surface with boundary.  The ``hole'' in the translation surface is the boundary of a cylinder so in the flat perspective we can regard it as a slit with no specified identification.

\begin{remark}
We remark that a similar surgery was already introduced in \cite[$\S$4]{AulicinoCompDegKZ}.  However, it is not sufficient here because it may not be well-defined in general.  In that surgery, the half-infinite cylinder was cut along its boundary instead of unit distance away from its boundary.  This could lead to the possibility that the resulting surface has zero area, and it leads to unnecessary ambiguities here that we wish to avoid.  The fact that we do not cut the cylinder along its boundary is the reason we use the term \emph{partial}.  Of course the choice of unit distance was arbitrary and any positive distance would suffice.
\end{remark}

\subsection{Surgery of a Higher Order Pole}

The partial pole surgery for higher order poles is in fact just the truncation procedure of \cite{GuptaMeroQuadHalfPlane}.  See \cite[Fig. 4]{GuptaMeroQuadHalfPlane} or \cite[$\S$ E.3]{FenyesAbelSL2RLocSys} for a flat picture of the differential where the higher order pole is at infinity, and \cite[Fig. 5]{GuptaMeroQuadHalfPlane} for a complex-analytic picture of the foliation, which corresponds to \cite[Figs. 12, 13]{Strebel}, where the higher order pole is at the origin.  Let $P$ be a point on $X$ at which $q$ has a pole of order at least three.  By \cite[Thm. 7.4]{Strebel}, there exists a simply connected open neighborhood $U \subset X$ of $P$ such that every trajectory in $U$ converges to $P$.  By taking $U$ sufficiently small so that it contains no singularity other than $P$, we have that $U$ is in the complement of every saddle connection on $(X, q)$ in every direction.  The \emph{truncation of height $H$} \cite{GuptaMeroQuadHalfPlane}, or the \emph{partial pole surgery}, is given by cutting along the half rectangles in \cite[Fig. 4]{GuptaMeroQuadHalfPlane}, and removing the infinite area region contained in the neighborhood $U$.  If $q$ has a pole of order $p$ at $P$, then the resulting slit translation surface has $2p-4$ slits on its boundary resulting from this surgery: $p-2$ slits in the vertical foliation and $p-2$ slits in the horizontal foliation.

\subsection{Applications of Surgery to Cantor-Bendixson Rank}
\label{AppsSurgCBRkSect}

We begin with a formal definition of a slit translation surface.

\begin{definition}
A \emph{slit translation surface} $(\hat X, \hat \omega)$ is a translation surface $(X, \omega)$ with (possibly empty) boundary and a finite collection of marked line segments $\{\sigma_1, \ldots, \sigma_n\}$ called \emph{slits}.  By convention, we assume that every boundary component of $(\hat X, \hat \omega)$ is contained in the set of slits.
\end{definition}

The motivation for the surgery is that we are interested in exactly the saddle connections on a Riemann surface that do not cross the slits because they correspond to trajectories that converge to a higher order pole.  Hence, a saddle connection that avoids a slit on $(\hat X, \hat q)$ avoids a higher order pole on $(X, q)$.  Furthermore, Lemma \ref{CBRankCov}, implies that it suffices to pass to Abelian differentials for the remainder of the paper as well.

Another difficulty we will encounter below is that when we wish to find an upper bound for the Cantor-Bendixson rank, we wish to use Lemma \ref{CantorBendClosedSetBd} to forget all but one slit.  However, this is meaningless when $(\hat X, \hat \omega)$ has boundary because there is no natural way to forget a boundary and get a Riemann surface without boundary.  Furthermore, there is no natural identification of the slits in the boundary because as observed in \cite{GuptaMeroQuadHalfPlane}, if a higher order pole has a non-zero ``metric residue,'' then to each slit there is a parallel slit, but it does not necessarily have the same length, so identification is impossible.  In spite of this difficulty, there is a simple remedy.  

\begin{definition}
Let $(\hat X, \hat \omega)$ be a slit translation surface with boundary.  Consider a second copy of it $(\hat X, -\hat \omega)$, which is rotated by $\pi$.  Then every boundary slit on $(\hat X, \hat \omega)$ trivially has a corresponding slit with the same length and orientation on $(\hat X, -\hat \omega)$.  Identify each boundary slit on $(\hat X, \hat \omega)$ with its copy on $(\hat X, -\hat \omega)$.  We call this procedure the \emph{doubling construction} and denote the resulting slit translation surface by $(\hat X, \hat \omega)^2$.
\end{definition}

The slits on $(\hat X, \hat \omega)^2$ are simply marked line segments.  Loosely speaking, they occur at the boundary of each copy of $(\hat X, \hat \omega)$ in $(\hat X, \hat \omega)^2$.  Now we prove that the doubling construction does indeed produce a translation surface.

\begin{lemma}
\label{DoubYieldsTranSurf}
If $(\hat X, \hat \omega)$ is a slit translation surface with boundary, then the doubling procedure produces a slit translation surface $(\hat X, \hat \omega)^2$ without boundary.
\end{lemma}

\begin{proof}
To prove this, we construct a half-translation surface without boundary and show that its canonical double cover exactly coincides with $(\hat X, \hat \omega)^2$.  Hence, $(\hat X, \hat \omega)^2$ must be a translation surface.

Given $(\hat X, \hat \omega)$, mark the midpoint of every slit that is in the boundary of $(\hat X, \hat \omega)$.  If there is no canonical midpoint, such as in the surgery of a double pole of a quadratic differential, then the slit is a closed curve; we mark any point and its antipode.  Consider the identification given by identifying each half of the slit to form a simple pole of a quadratic differential at the midpoint of every slit in a boundary.  For a slit arising from the double pole of a quadratic differential, we also get a simple pole at the antipode.  Let $A$ denote the resulting flat surface.

We claim that $A$ does indeed carry a meromorphic quadratic differential with at most simple poles.  Observe the existence of canonical measured foliations on the interior of $A$ induced by $\hat \omega$.  Since the angles about the midpoint of every slit is $\pi$, every foliation induced by $\hat \omega$ remains a foliation on $A$, though it is no longer orientable.  Thus, $A$ admits a quadratic differential by \cite{HubbardMasur}.

Finally, we observe that the double cover applied to $A$ is ramified over the midpoint (and antipode) of every slit.  This results in every boundary slit being identified to its copy rotated by $\pi$ in the definition of the canonical double covering construction, which is also exactly the identification in the doubling construction.  Hence, these constructions yield the same surface, which consequentially must be a translation surface.
\end{proof}

Now we explain the conventions that we use to extend some standard terminology to slit translation surfaces.  There are two conventions we introduce for slit translation surfaces depending on the context in which they arise.  Given a translation surface $(X, \omega)$, possibly with marked points, let $\Sigma$ denote the set of points in $X$ that are either marked or at which $\omega$ has a singularity.

\begin{convention}
\label{SlitTransConv}
Given a \emph{slit translation surface}, $\Sigma$ includes the endpoints of every slit.
\end{convention}

\begin{convention}
\label{MeromorConv}
Given a \emph{meromorphic quadratic differential} $(X, q)$ and letting $(\hat X, \hat \omega)$ be the slit translation surface arising from applying the surgery to $(X, q)$, we do \emph{not} include the endpoints of the slits of $(\hat X, \hat \omega)$ in $\Sigma$.
\end{convention}

The reason for the two different conventions is that, by including the endpoints of the slit, we are forced to consider saddle connections that emanate or terminate at these endpoints.  However, if the object of study is a meromorphic quadratic differential, then the slits that are introduced by the surgery are artificial in nature and including endpoints of them would introduce saddle connections on the associated slit translation surface that do not correspond to saddle connections on the meromorphic quadratic differential.  This distinction is important in Section \ref{PfGenGMNLemSect}, where we will primarily be studying meromorphic quadratic differentials.  Sections \ref{SnExSect} and \ref{BdCBRankSect}, on the other hand, will use Convention \ref{SlitTransConv} where the endpoints of the slits are included.

\begin{definition}
Let $(\hat X, \hat \omega)$ be a slit translation surface.  Let $\cF_{\theta}$ denote the vertical foliation of $(X, e^{i\theta}q)$.  Define the set
$$\Theta(\hat X, \hat \omega) = \left \{ \theta \in [0,\pi) \middle \vert  \begin{array}{l}\cF_{\theta} \text{ admits a saddle connection whose interior } \\
 \text{is disjoint from the interior of every slit} \end{array} \right\}.$$
By convention, $\Theta(\hat X, \hat \omega)$ will always include the finite set of directions in which the slits lie.
\end{definition}

In simpler terms, we only consider saddle connections on slit translation surfaces that do not cross any of the slits.

\begin{lemma}
\label{DoubSameCBRk}
If $(\hat X, \hat \omega)$ is a slit translation surface with boundary, then 
$$\Theta(\hat X, \hat \omega)^2 = \Theta(\hat X, \hat \omega).$$
\end{lemma}

\begin{proof}
By definition of a saddle connection on a slit translation surface, $\Theta(\hat X, \hat \omega)^2 \subseteq \Theta(\hat X, \hat \omega)$.  On the other hand, since the direction of each saddle connections in each copy of $(\hat X, \hat \omega)$ is preserved under the doubling construction, equality of the sets follows.
\end{proof}

Thus, it suffices to restrict our attention to slit translation surfaces without boundary.  We make one further simplification of the setting at hand.  Given a slit translation surface without boundary, we can forget all but one slit.  By Lemma \ref{CantorBendClosedSetBd}, any upper bound on the Cantor-Bendixson rank of the set of saddle connection directions on a slit translation surface with one slit automatically implies an upper bound on the Cantor-Bendixson rank of the set of saddle connection directions on a slit translation surface with arbitrarily many slits.  For this reason, we switch perspective for the remainder of the paper and devote our study to finite area translation surfaces with slits.

\section{Slit Tori}

In this section we focus on a torus carrying an Abelian differential with a collection of slits.  However, we collect general results throughout.  From this perspective, the torus can be realized as a parallelogram (or a square without loss of generality) with a collection of slits.  We show that depending on the arrangement of the slits the Cantor-Bendixson rank can be one, two, or three, and it can never be greater than three.

\begin{warning}
The tori in this section arise exclusively from Abelian differentials and \emph{not} from quadratic differentials.  In general, a torus carrying a finite area quadratic differential cannot be realized as a parallelogram.  See \cite{AthreyaEskinZorichCountGenJSDiffs} for examples of how complicated the flat geometry of a quadratic differential on a sphere could be.
\end{warning}

\subsection{Cantor-Bendixson Rank $1$}

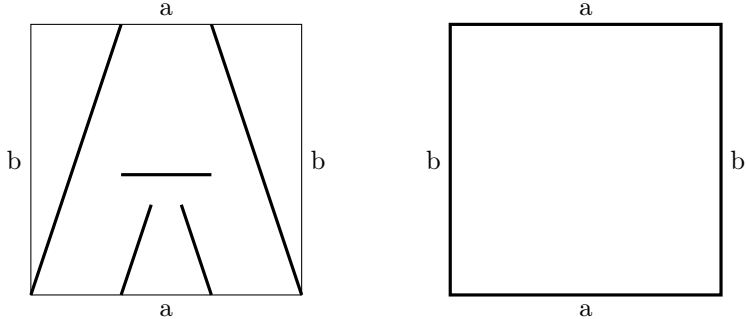
\begin{figure}[h!]
\centering
\begin{minipage}[t]{0.45\linewidth}
\centering
\begin{tikzpicture}[scale=0.40]
\draw (0,0) rectangle (9,9);
\draw [very thick] (0,0)--(3,9);
\draw [very thick] (9,0)--(6,9);
\draw [very thick] (3,0)--(4,3);
\draw [very thick] (6,0)--(5,3);
\draw [very thick] (3,4)--(6,4);
\draw(4.5,9) node[above] {a};
\draw(0,4.5) node[left] {b};
\draw(4.5,0) node[below] {a};
\draw(9,4.5) node[right] {b};
\end{tikzpicture}
\end{minipage}
\begin{minipage}[b]{0.45\linewidth}
\centering
\begin{tikzpicture}[scale=0.40]
\draw [very thick] (0,0) rectangle (9,9);
\draw(4.5,9) node[above] {a};
\draw(0,4.5) node[left] {b};
\draw(4.5,0) node[below] {a};
\draw(9,4.5) node[right] {b};
\end{tikzpicture}
\end{minipage}
\caption{A torus with three slits marked by thickened line segments (left), and a torus with two coinciding slits (right)}
\label{Rank1ExFig}
\end{figure}

\begin{lemma}
\label{FinCyl}
If a meromorphic Abelian differential contains a cylinder disjoint from the slits, then $\Theta(\hat X, \hat \omega)$ is infinite.
\end{lemma}

\begin{proof}
If $(\hat X, \hat \omega)$ contains a finite cylinder, then there is a zero on its top boundary and a zero on its bottom boundary.  Since the cylinder is disjoint from the slits, there are infinitely many trajectories (saddle connections) between these two (not necessarily distinct) zeros that correspond to winding around the cylinder.  Thus, the set of directions that admit saddle connections is infinite.
\end{proof}

Lemma \ref{FinCyl} implies that a prerequisite to finding an example of a slit translation surface such that $\Theta(\hat X, \hat \omega)$ is finite is the absence of finite cylinders disjoint from the slits.

\begin{lemma}
\label{FigHasCB1}
Let $(\hat X,\hat \omega)$ denote either torus depicted in Figure \ref{Rank1ExFig}.  Then $\Theta(\hat X, \hat \omega)$ is finite, and in particular, $\Theta(\hat X, \hat \omega)$ has Cantor-Bendixson rank one.
\end{lemma}

\begin{proof}
Since no trajectory can wind around the surface without passing through a slit, any saddle connection must have length bounded above.  However, there are only finitely many trajectories between marked points on any surface with bounded length.  Hence, $\Theta(\hat X, \hat \omega)$ is finite.
\end{proof}

\subsection{Cantor-Bendixson Rank $2$}

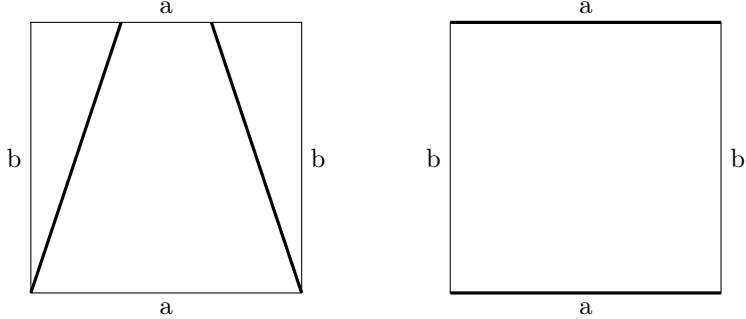
\begin{figure}[h!]
\centering
\begin{minipage}[t]{0.45\linewidth}
\centering
\begin{tikzpicture}[scale=0.40]
\draw (0,0) rectangle (9,9);
\draw [very thick] (0,0)--(3,9);
\draw [very thick] (9,0)--(6,9);
\draw(4.5,9) node[above] {a};
\draw(0,4.5) node[left] {b};
\draw(4.5,0) node[below] {a};
\draw(9,4.5) node[right] {b};
\end{tikzpicture}
\end{minipage}
\begin{minipage}[b]{0.45\linewidth}
\centering
\begin{tikzpicture}[scale=0.40]
\draw (0,0) rectangle (9,9);
\draw [very thick] (0,0)--(9,0);
\draw [very thick] (0,9)--(9,9);
\draw(4.5,9) node[above] {a};
\draw(0,4.5) node[left] {b};
\draw(4.5,0) node[below] {a};
\draw(9,4.5) node[right] {b};
\end{tikzpicture}
\end{minipage}
\caption{A torus with two slits marked by thickened line segments (left), and a torus with one slit with coinciding endpoints (right)}
\label{Rank2ExFig}
\end{figure}

\begin{lemma}
\label{TorusCylAccSlitDir}
If $(\hat X, \hat \omega)$ is a torus with a slit, then the accumulation point of the set of cylinder directions of cylinders disjoint from the slit must lie exactly in the direction of the slit.
\end{lemma}

\begin{proof}
By contradiction, assume that there is an accumulation point $\theta' \in \Theta(\hat X, \hat \omega)$ of directions of cylinders in a direction distinct from that in which a slit lies on the torus.  This implies that there is a sequence of directions containing cylinders $\{\theta_n\}$ converging to $\theta'$.  Let $s$ be the slit on the torus, which by assumption does not lie in direction $\theta'$.  In particular, if we consider a line segment $\sigma$ of unit length in the direction perpendicular to $\theta'$, then the projection of $s$ onto $\sigma$ yields a line segment $s'$ with positive length.  Furthermore, as $n$ goes to infinity, the lengths of the closed curves of cylinders in direction $\theta_n$ tend to infinity, and their intersections with $\sigma$ become dense in $\sigma$.  Hence, for $N$ sufficiently large, the target distance between the points of intersection of the closed trajectory in direction $\theta_N$ and $\sigma$ is less than the length of $s'$, which implies that there is no cylinder on the torus in direction $\theta_N$ disjoint from the slit $s$.  This contradiction completes the proof of the lemma.
\end{proof}

Let $\{s_n\}$ be an infinite sequence of saddle connections on $(X, \omega)$.  Define a subset of the flat surface by
$$\Omega(\{s_n\}) := \bigcap_{n = 1}^{\infty}\overline{\bigcup_{k \geq n}s_k}.$$
We will suppress the sequence when it is understood.  Observe that by the compactness of $X$, $\Omega(\{s_n\}) \not = \emptyset$.

\begin{lemma}
\label{InfSCImpFinCyl}
Given an infinite sequence of saddle connections $\{s_n\}$ with a single accumulation point of directions $\theta'$ on a slit translation surface $(\hat X, \hat \omega)$, the set $\Omega(\{s_n\})$ is a finite union of parallel invariant components of $(\hat X, \hat \omega)$ in the direction $\theta'$.  Furthermore, there exists a finite cylinder $C \subset \Omega(\{s_n\}) \subset (\hat X, \hat \omega)$ disjoint from the slits that does not necessarily lie in the direction $\theta'$.
\end{lemma}

We postpone the proof of Lemma \ref{InfSCImpFinCyl} to Section \ref{PfGenGMNLemSect}.  See Figure \ref{Rank2ExFig} (left) for an example of a slit torus satisfying the assumptions of the Proposition \ref{Tor2SlitsCBR2}.

\begin{proposition}
\label{Tor2SlitsCBR2}
The set $\Theta$ for a torus with two non-parallel slits and a finite cylinder disjoint from the slits has Cantor-Bendixson rank two.
\end{proposition}

\begin{proof}
Since the slit torus contains a finite cylinder disjoint from the cylinders, $\Theta$ contains infinitely many elements by Lemma \ref{FinCyl}, thus the Cantor-Bendixson rank is at least two.

By Lemma \ref{TorusCylAccSlitDir}, the accumulation points of $\Theta$ corresponding to closed curves of cylinders can only lie in the direction of a slit.  However, if the slits do not lie in the same direction, then the subset of $\Theta$ corresponding to closed curves of cylinders must be finite because neither direction of a slit can be an accumulation point of cylinder directions due to the presence of the other slit and Lemma \ref{TorusCylAccSlitDir}.

By Lemma \ref{InfSCImpFinCyl}, every infinite sequence of saddle connections implies the existence of a finite cylinder in which a subset of the infinite sequence is dense.  Since there are only finitely many finite cylinders, consider the collection of all saddle connections contained in each of the finite cylinders, which have finitely many accumulation points corresponding to the core curves of the finite cylinders.  We claim there cannot be an infinite sequence of saddle connections in the complement of these saddle connections.  By contradiction, if there were, then they would accumulate to a finite cylinder $C$ that contains an infinite sequence of them by Lemma \ref{InfSCImpFinCyl}.  However, the saddle connections were chosen to be outside of a finite collection of cylinders on the surface, which yields a contradiction.  Hence, $\Theta$ has Cantor-Bendixson rank at most two if there are two slits lying in distinct directions.
\end{proof}

\subsection{Cantor-Bendixson Rank $3$}

\begin{lemma}
\label{TorCBRMin3}
The Cantor-Bendixson rank of the set of saddle connections on the torus with one slit with distinct end points is at least three.
\end{lemma}

\begin{proof}
Without loss of generality, we regard the torus as the unit square in the plane.  The slit has either rational or irrational slope.

\

\noindent \textit{Rational Case:} Consider a set of directions on the torus converging to the direction in which the slit lies.  Let the slit lie in the rational direction with slope $p/q$.  After applying a linear transformation and cutting and regluing the torus, it suffices to consider the case of the torus as a unit square in the plane with a slit of length less than one lying along the $x$-axis without loss of generality.  Then in the direction of slope $1/n$, it is clear that there is one cylinder, which is bounded by the endpoints of the slit, and any closed trajectory in the complement of the cylinder must pass through the slit.  By Lemma \ref{FinCyl}, each cylinder gives rise to a distinct infinite sequence of saddle connection directions converging to the direction of its core curve, and the directions of the core curves of the cylinders themselves converge to the horizontal direction.  Hence, the Cantor-Bendixson rank is at least three if the slit has rational slope.

\

\noindent \textit{Irrational Case:} Let the slit have slope $\alpha \in \bR \setminus \bQ$.  It suffices to consider $\alpha < 1$.  Let $p_n/q_n$ be the sequence of continued fraction approximates converging to $\alpha$.  Let $L$ be the length of the slit.  Choose $N$ sufficiently large so that $\sqrt{p_n^2+q_n^2} > L$ for all $n \geq N$.  Let $\theta_n$ be the angle between the line of slope $p_n/q_n$ and $\alpha$.  Then the length of the projection of the slits onto the orthogonal line of slope $-q_n/p_n$ is given by $L \sin \theta_n$.  The torus in direction $p_n/q_n$ can be realized as a rectangle with sides of lengths $\sqrt{p_n^2+q_n^2}$ and $1/\sqrt{p_n^2+q_n^2}$.  In order to see the existence of a cylinder disjoint from the slits it suffices to show that the projection of the slits onto the short side of the rectangle is less than the length of the short side of the rectangle, i.e.
$$L\sin \theta_n < \frac{1}{\sqrt{p_n^2+q_n^2}},$$
or equivalently $L^2 \sin^2(\theta_n) (p_n^2 + q_n^2) < 1$.  By Dirichlet's Approximation Theorem, there are infinitely many $n$ satisfying the inequality
$$\left| \alpha - \frac{p_n}{q_n}\right| < \frac{1}{q_n^2}.$$
Let $\tilde \theta_n$ be the angle of the line with slope $p_n/q_n$ so that $\tan \tilde \theta_n = p_n/q_n$.  Let $\tan \theta' = \alpha$.  Using $\tan \theta_n = \tan (\theta' - \tilde \theta_n)$ yields
$$\sin \theta_n = \frac{\alpha - \frac{p_n}{q_n}}{\sqrt{\left(\alpha - \frac{p_n}{q_n}\right)^2 + \left(1 + \alpha\frac{p_n}{q_n}\right)^2}}.$$
Then this implies
$$L^2 \sin^2(\theta_n) (p_n^2 + q_n^2) = \frac{L^2(p_n^2 + q_n^2)\left(\alpha - \frac{p_n}{q_n}\right)^2}{\left(\alpha - \frac{p_n}{q_n}\right)^2 + \left(1 + \alpha\frac{p_n}{q_n}\right)^2}$$
$$< L^2(p_n^2 + q_n^2)\frac{1}{q_n^4} = \frac{L^2}{q_n^2}\left(\frac{p_n^2}{q_n^2} + 1\right) < \frac{2L^2}{q_n^2},$$
which clearly converges to zero as $n$ tends to infinity.
\end{proof}

\begin{proposition}
\label{TorCBR3Prop}
The Cantor-Bendixson rank of the set of saddle connections on the torus with one slit with distinct endpoints is three.
\end{proposition}

\begin{proof}
By Lemma \ref{TorCBRMin3}, it suffices to prove that the Cantor-Bendixson rank is bounded above by three.  By the definition of a saddle connection, every saddle connection on the torus with one slit must begin and terminate at the end points of the slit.  If a saddle connection emanates and terminates from the same endpoint of the slit, it determines a closed trajectory, whence it determines a cylinder and these saddle connections can only accumulate in the direction of the slit by Lemma \ref{TorusCylAccSlitDir}.  If a sequence of saddle connections $\{s_n\}$ connects distinct endpoints of the slit, then $\Omega(\{s_n\})$ is an invariant component of the torus by Lemma \ref{InfSCImpFinCyl}.  Hence, it can only be a minimal component in the direction of the slit or a cylinder disjoint from the slit.  Again the cylinder directions can only accumulate to the direction of the slit by Lemma \ref{TorusCylAccSlitDir}.  Therefore, sequences of saddle connections can accumulate to the direction of the slit or a finite cylinder, and those finite cylinders can only accumulate to the direction of the slit.  Thus, the Cantor-Bendixson rank must be exactly three in this case.
\end{proof}

\section{Proof of the Generalized GMN Lemma}
\label{PfGenGMNLemSect}

In this section we prove Theorem \ref{DoubPoleClosed}.  We begin with the proof of Lemma \ref{InfSCImpFinCyl}.

\begin{proof}[Proof of Lemma \ref{InfSCImpFinCyl}]
Without loss of generality, we pass to the following subsequence of $\{s_n\}$.  Since saddle connections connect pairs of singularities, and there are finitely many singularities, it suffices to pass to a subsequence such that $s_n$ connects the same two (not necessarily distinct) singularities for all $n$.  Secondly, it suffices to pass to another subsequence so that the corresponding sequence of angles $\{\theta_n\}$ converges monotonically to a fixed angle $\theta'$.  It is clear that there are at most finitely many saddle connections of length at most $L$ on a flat surface.  Hence, the flat lengths of the saddle connections in the sequence $\{s_n\}$ tend to infinity with $n$.  Let $\Omega$ denote the set defined with respect to this subsequence.

Let $z_0$ denote one of the singularities from which all of the saddle connections $s_n$ emanate.  Since the sequence of saddle connections is converging monotonically in angle to the limiting direction $\theta'$, the trajectory in direction $\theta'$ emanating from $z_0$ is contained in $\Omega$.  If this trajectory is semi-infinite, then its closure is a union of invariant components (See \cite[$\S$3.4]{VianaIETBook}).  Otherwise, it terminates at a zero $z_1$, forming a saddle connection $\sigma_1$, which is contained in $\Omega$ by definition.  Consider the trajectory emanating from $z_1$ in direction $\theta'$, which must be contained in $\Omega$.  Again it is either semi-infinite, and we conclude as before, or it terminates at a zero $z_2$, forming a saddle connection $\sigma_2$, and we continue this procedure.  Since there are only finitely many saddle connections in a given direction, this process must terminate in either a semi-infinite trajectory or at the copy of the zero $z_0$ incident with the saddle connection $\sigma_1$.  In the latter case, we see that the union of all of the saddle connections form the boundary of a cylinder $C$, by construction.  Since the saddle connections $\{s_n\}$ become dense in $C$ as $n$ tends to infinity, we get that $C \subseteq \Omega$.

If $\Omega$ is a union of invariant components in direction $\theta'$ such that none of them are cylinders, then after passing to a connected component (which is in fact unnecessary because it can be easily proven that $\Omega$ is connected), $\Omega$ is a Riemann surface with boundary.  Furthermore, the boundary curves lie in the same foliation on $(\hat X, \hat \omega)$ corresponding exactly to $\theta'$, and by construction the interiors of the invariant components are disjoint from the slits.  The existence of the finite cylinder $C \subseteq \Omega$ disjoint from the slits follows from \cite[Cor. 4.4]{AulicinoCompDegKZ}.
\end{proof}

\begin{proof}[Proof of Theorem \ref{DoubPoleClosed}]
Let $(X, q)$ be the Riemann surface carrying the meromorphic differential.  Consider its canonical double cover $(\tilde X, \omega)$ and recall that $\Theta(X, q) = \Theta(\tilde X, \omega)$ by Lemma \ref{CBRankCov}.  Perform the surgery on $(\tilde X, \omega)$ to get a slit translation surface $(\hat X, \hat \omega)$.  Next, we claim that $\Theta(\tilde X, \omega) = \Theta(\hat X, \hat \omega)$.  By Convention \ref{MeromorConv}, endpoints of slits arising from the surgery of a meromorphic quadratic (or Abelian) differential are not included in the set of marked points, we have that the set of marked points of $(\tilde X, \omega)$ is the same as the set of marked points on $(\hat X, \hat \omega)$.  Since the surfaces $(\tilde X, \omega)$ and $(\hat X, \hat \omega)$ agree up to neighborhoods of the poles, which do not intersect any saddle connection by construction, the claim follows.

We consider several cases.  If there are only finitely many saddle connections, then the set of saddle connection directions is finite, thus it is closed. If there are infinitely many saddle connections, then after passing to a sufficient subsequence of saddle connections, they must accumulate in a union of invariant components $\Omega$ by Lemma \ref{InfSCImpFinCyl}.  If $\Omega$ consists of two or more invariant components, then the boundary between them contains a saddle connection and we are done.  If $\Omega$ is a cylinder, then its boundary contains saddle connections, and again we are done.  If $\Omega$ is a minimal component that is a proper subset of the surface, then it has a boundary consisting of a union of saddle connections.  

Finally, we consider the case where $\Omega$ is a minimal component that is equal to the entire surface.  The first observation is that if $q$ has a pole of order greater than two, then there are slits in two distinct directions by definition of the partial pole surgery, thus $\Omega$ must be a proper subset of the surface.  Hence, if $\Omega$ is equal to $(\hat X, \hat \omega)$ up to a set of measure zero, then every pole of $q$ has order at most two.  Moreover, each slit lies exactly in the direction of the closed trajectories of the half-infinite cylinder that was removed by definition of the surgery.  The boundary of the half-infinite cylinder on $q$ which is now exactly the direction of $\Omega$ consists of a union of saddle connections.  Thus, the limit of every convergent sequence of saddle connection directions is a direction in which the foliation on the surface admits a saddle connection.
\end{proof}

From the definition of $\Theta(\hat X, \hat \omega)$ for a slit translation surface and the proof of Theorem \ref{DoubPoleClosed}, the following corollary is immediate.

\begin{corollary}
\label{SlitTransClsdSCDirs}
If $(\hat X, \hat \omega)$ is a slit translation surface (possibly with boundary), then $\Theta(\hat X, \hat \omega)$ is closed.
\end{corollary}

\section{Examples}
\label{SnExSect}

In this section we present an infinite family of slit translation surfaces realizing arbitrarily large Cantor-Bendixson rank as the genus tends to infinity.  The family is depicted in Figure \ref{FamCBRkn}. In each figure the saddle connection $a$ in bold represents a slit, and it is identified to the other copy of $a$.  Denote the $n$'th member of the family by $\hat S_n$.  Then $\hat S_n$ consists of $n+1$ squares and lies in a stratum of complex dimension $n+2$.  The reader can check that $\hat S_n \in \cH(n)$ if $n$ is even, and $\hat S_n \in \cH(\frac{n-1}{2}, \frac{n-1}{2})$ if $n$ is odd.

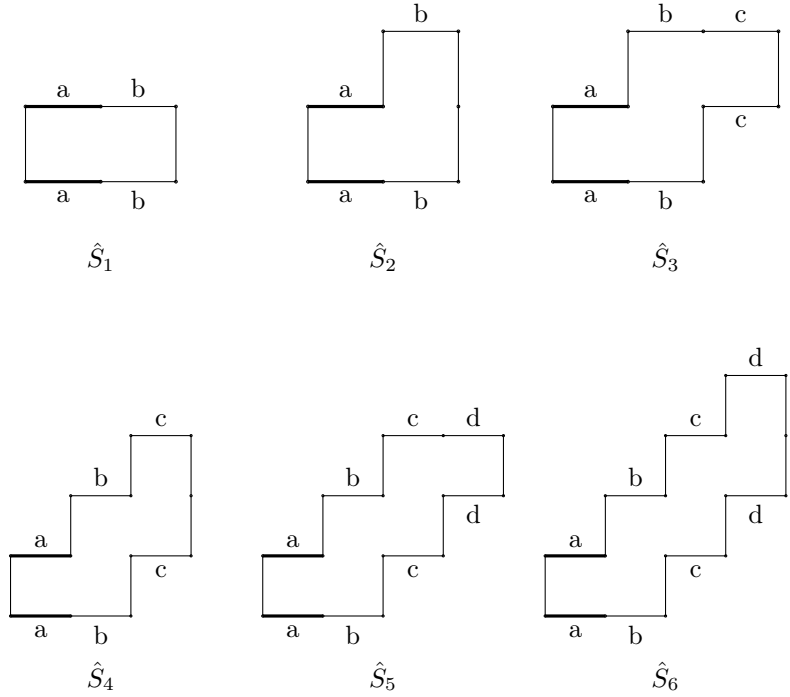
\begin{figure}[h!]
\begin{minipage}[t]{0.3\linewidth}
\centering
\begin{tikzpicture}[scale=0.50]
\draw (0,0)--(0,2)--(4,2)--(4,0)--cycle;
\draw [very thick] (0,0)--(2,0);
\draw [very thick] (0,2)--(2,2);
\foreach \x in {(0,0),(0,2),(2,2),(4,2),(4,0),(2,0)} \draw \x circle (1pt);
\draw(1,2) node[above] {a};
\draw(3,2) node[above] {b};
\draw(1,0) node[below] {a};
\draw(3,0) node[below] {b};
\draw(2,-2) node {$\hat S_1$};
\end{tikzpicture}
\end{minipage}
\begin{minipage}[t]{0.3\linewidth}
\centering
\begin{tikzpicture}[scale=0.50]
\draw (0,0)--(0,2)--(2,2)--(2,4)--(4,4)--(4,2)--(4,0)--cycle;
\draw [very thick] (0,0)--(2,0);
\draw [very thick] (0,2)--(2,2);
\foreach \x in {(0,0),(0,2),(2,2),(2,4),(4,4),(4,2),(4,0),(2,0)} \draw \x circle (1pt);
\draw(1,2) node[above] {a};
\draw(3,4) node[above] {b};
\draw(1,0) node[below] {a};
\draw(3,0) node[below] {b};
\draw(2,-2) node {$\hat S_2$};
\end{tikzpicture}
\end{minipage}
\begin{minipage}[b]{0.3\linewidth}
\centering
\begin{tikzpicture}[scale=0.50]
\draw (0,0)--(0,2)--(2,2)--(2,4)--(4,4)--(6,4)--(6,2)--(4,2)--(4,0)--cycle;
\draw [very thick] (0,0)--(2,0);
\draw [very thick] (0,2)--(2,2);
\foreach \x in {(0,0),(0,2),(2,2),(2,4),(4,4),(6,4),(6,2),(4,2),(4,0),(2,0)} \draw \x circle (1pt);
\draw(1,2) node[above] {a};
\draw(3,4) node[above] {b};
\draw(5,4) node[above] {c};
\draw(1,0) node[below] {a};
\draw(3,0) node[below] {b};
\draw(5,2) node[below] {c};
\draw(3,-2) node {$\hat S_3$};
\end{tikzpicture}
\end{minipage}

\vspace{0.75cm}

\begin{minipage}[t]{0.3\linewidth}
\centering
\begin{tikzpicture}[scale=0.40]
\draw (0,0)--(0,2)--(2,2)--(2,4)--(4,4)--(4,6)--(6,6)--(6,4)--(6,2)--(4,2)--(4,0)--cycle;
\draw [very thick] (0,0)--(2,0);
\draw [very thick] (0,2)--(2,2);
\foreach \x in {(0,0),(0,2),(2,2),(2,4),(4,4),(4,6),(6,6),(6,4),(6,2),(4,2),(4,0),(2,0)} \draw \x circle (1pt);
\draw(1,2) node[above] {a};
\draw(3,4) node[above] {b};
\draw(5,6) node[above] {c};
\draw(1,0) node[below] {a};
\draw(3,0) node[below] {b};
\draw(5,2) node[below] {c};
\draw(3,-2) node {$\hat S_4$};
\end{tikzpicture}
\end{minipage}
\begin{minipage}[b]{0.3\linewidth}
\centering
\begin{tikzpicture}[scale=0.40]
\draw (0,0)--(0,2)--(2,2)--(2,4)--(4,4)--(4,6)--(6,6)--(8,6)--(8,4)--(6,4)--(6,2)--(4,2)--(4,0)--cycle;
\draw [very thick] (0,0)--(2,0);
\draw [very thick] (0,2)--(2,2);
\foreach \x in {(0,0),(0,2),(2,2),(2,4),(4,4),(4,6),(6,6),(8,6),(8,4),(6,4),(6,2),(4,2),(4,0),(2,0)} \draw \x circle (1pt);
\draw(1,2) node[above] {a};
\draw(3,4) node[above] {b};
\draw(5,6) node[above] {c};
\draw(7,6) node[above] {d};
\draw(1,0) node[below] {a};
\draw(3,0) node[below] {b};
\draw(5,2) node[below] {c};
\draw(7,4) node[below] {d};
\draw(4,-2) node {$\hat S_5$};
\end{tikzpicture}
\end{minipage}
\begin{minipage}[b]{0.3\linewidth}
\centering
\begin{tikzpicture}[scale=0.40]
\draw (0,0)--(0,2)--(2,2)--(2,4)--(4,4)--(4,6)--(6,6)--(6,8)--(8,8)--(8,4)--(6,4)--(6,2)--(4,2)--(4,0)--cycle;
\draw [very thick] (0,0)--(2,0);
\draw [very thick] (0,2)--(2,2);
\foreach \x in {(0,0),(0,2),(2,2),(2,4),(4,4),(4,6),(6,6),(6,8),(8,8),(8,6),(8,4),(6,4),(6,2),(4,2),(4,0),(2,0)} \draw \x circle (1pt);
\draw(1,2) node[above] {a};
\draw(3,4) node[above] {b};
\draw(5,6) node[above] {c};
\draw(7,8) node[above] {d};
\draw(1,0) node[below] {a};
\draw(3,0) node[below] {b};
\draw(5,2) node[below] {c};
\draw(7,4) node[below] {d};
\draw(4,-2) node {$\hat S_6$};
\end{tikzpicture}
\end{minipage}
\caption{The Family of Surfaces $\hat S_n$}
\label{FamCBRkn}
\end{figure}

\begin{proposition}
\label{CBRankFam}
The Cantor-Bendixson rank of $\Theta(\hat S_n)$ is at least $n+2$.
\end{proposition}

\begin{proof}
We proceed by induction.  When $n=1$, the first element of the family is a slit torus, so the Cantor-Bendixson rank of $\Theta(\hat S_1)$ is at least three by Proposition \ref{TorCBR3Prop}.  The induction hypothesis assumes that $\Theta(\hat S_{n-1})$ has Cantor-Bendixson rank $n+1$.  We claim that, roughly speaking, $\hat S_n$ contains infinitely many ``copies'' of $\hat S_{n-1}$ that avoid the slit $a$.  Observe that there is an infinite sequence of cylinders contained in the bottom cylinder in $\hat S_n$ that contain the slit $a$ that are formed by considering trajectories from $a$ to itself, e.g. the white cylinder in Figure \ref{PfofPropCBRkFamFig}.  The infinite sequence of directions corresponding to these cylinders converges to the horizontal direction.  Consider the complement of each of these cylinders.  Since all horizontal cylinders consist of two squares with the possible exception of the top cylinder, which may only consist of one square, we get that the complement of the cylinder containing $a$ consists of a cylinder from $b$ to itself, $c$ to itself, etc. until we reach the top square, which must be entirely contained in a cylinder.  Then by marking the boundary of the cylinder from $b$ to itself along the boundary that is incident with the cylinder containing $a$, we see a stretched and sheared copy of $\hat S_{n-1}$ contained in $\hat S_n$, e.g. the union of shaded regions in Figure \ref{PfofPropCBRkFamFig}.  By the induction hypothesis and the existence of an infinite sequence of $\hat S_{n-1}$ converging in $\hat S_n$ to the horizontal direction, the Cantor-Bendixson rank of $\Theta(\hat S_n)$ must be at least one more than that of $\Theta(\hat S_{n-1})$.  Thus, it is at least $n+2$.
\end{proof}

\begin{figure}
\centering
\begin{tikzpicture}[scale=0.75]
\draw [thick] (0,0)--(0,4)--(4,4)--(4,8)--(8,8)--(12,8)--(12,4)--(8,4)--(8,0)--cycle;
\draw [thick] (4,0)--(8,1);
\draw [thick] (0,0)--(8,2);
\draw [thick] (0,1)--(8,3);
\draw [thick] (0,2)--(8,4);
\draw [thick] (0,3)--(4,4);
\draw [thick] (0,2)--(12,5);
\draw [thick] (0,3)--(12,6);
\draw [thick] (4,5)--(12,7);
\draw [thick] (4,6)--(12,8);
\draw [thick] (4,7)--(8,8);
\draw [very thick] (0,0)--(4,0);
\draw [very thick] (0,4)--(4,4);
\foreach \x in {(0,0),(0,4),(4,4),(4,8),(8,8),(12,8),(12,4),(8,4),(8,0)} \draw \x circle (1pt);
\draw(2,4) node[above] {a};
\draw(6,8) node[above] {b};
\draw(10,8) node[above] {c};
\draw(2,0) node[below] {a};
\draw(6,0) node[below] {b};
\draw(10,4) node[below] {c};
\path [fill=gray, fill opacity=.4] (4,0) -- (8,1) -- (8,0);
\path [fill=gray, fill opacity=.4] (0,0) -- (8,2) -- (8,3) -- (0,1);
\path [fill=gray, fill opacity=.4] (0,2) -- (12,5) -- (12,6) -- (0,3);
\path [fill=gray, fill opacity=.4] (4,5) -- (12,7) -- (12,8) -- (4,6);
\path [fill=gray, fill opacity=.4] (4,7) -- (8,8) -- (4,8);
\path [fill=blue, fill opacity=.4] (8,4) -- (12,5) -- (12,4);
\path [fill=blue, fill opacity=.4] (4,4) -- (12,6) -- (12,7) -- (4,5);
\path [fill=blue, fill opacity=.4] (4,6) -- (12,8) -- (8,8) -- (4,7);
\end{tikzpicture}
\caption{An example of $\hat S_3$ containing $\hat S_2$: The two cylinders of $\hat S_2$ are colored gray and blue}
\label{PfofPropCBRkFamFig}
\end{figure}
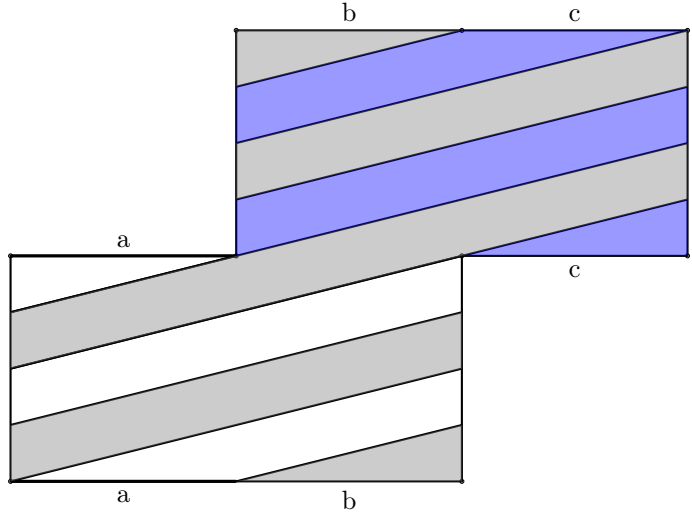

\section{Bounds on the Cantor-Bendixson Rank}
\label{BdCBRankSect}

We consider the concept of dimension for a space of translation surfaces with respect to individual translation surfaces so that we can define dimension for invariant components contained in a translation surface.  This quantity will be used in an induction argument to prove the main theorem.

Recall that $\Sigma$ is the union of the finite subset of points on $\hat X$ at which $\hat \omega$ has a singularity and the subset of points that are endpoints of slits.

\begin{definition}
Let $\Omega$ be a union of parallel invariant components in direction $\theta$ on a slit translation surface $(\hat X, \hat \omega)$.  Let $\Sigma' = \Omega \cap \Sigma$.  Define the dimension $d(\Omega)$ to be
$$d(\Omega) := \dim_{\bC} H_1(\Omega, \Sigma', \bC).$$
\end{definition}

It follows from the inclusion map $\iota: (\Omega, \Sigma') \hookrightarrow (\hat X, \Sigma)$, that there is an injection of homology $\iota_*: H_1(\Omega, \Sigma') \hookrightarrow H_1(\hat X, \Sigma)$, which implies that 
$$d(\Omega) \leq \dim_{\bC} \cH(\kappa).$$

We consider local period coordinates on a stratum relative to this subsurface.  Let $\{\gamma_1, \ldots, \gamma_m\}$ be a basis for $H_1(\Omega, \Sigma', \bZ)$.  The injection $\iota_*$ of first homology implies that this basis can be extended to a basis $\{\gamma_1, \ldots, \gamma_m, \gamma_{m+1} \ldots, \gamma_n \}$ for $H_1(\hat X, \Sigma, \bZ)$.  Then the period map
$$\Phi: (X, \omega) \mapsto \left( \int_{\gamma_1} \omega, \ldots, \int_{\gamma_n} \omega \right)$$
relative to this choice of basis provides coordinates for a neighborhood of $(X, \omega)$ as usual.  However, this choice has the additional property that the first $m$ coordinates completely parametrize the subsurface $\Omega$.  In particular, there is a local subspace given by fixing the first $m$ coordinates and letting the last $n-m$ coordinates vary that completely describes all local deformations that preserve $\Omega$.  Finally, we observe that if $(X, \omega) \in \cH(\kappa)$ is a translation surface, then regarding $(X, \omega)$ as a union of parallel invariant components in the vertical direction yields $d(X, \omega) = \dim_{\bC} \cH(\kappa)$.

\begin{lemma}
\label{LemInvCompImpSmallDim}
Let $(\hat X, \hat \omega)$ be a slit translation surface without boundary.  Let $\Omega$ be a union of parallel invariant components in direction $\theta$ contained in $(\hat X, \hat \omega)$ and disjoint from all of the slits of $(\hat X, \hat \omega)$.  If $\Omega' \subsetneq \Omega$ is a union of invariant components in direction $\theta' \not= \theta$, then $d(\Omega') < d(\Omega)$.
\end{lemma}

\begin{proof}
By the assumption that $\Omega' \subset \Omega$, it is clear that $d(\Omega') \leq d(\Omega)$ because linear independence is relative to the same surface $(\hat X, \hat \omega)$ for both sets of invariant components.  Let $\Sigma'' = \Sigma \cap \Omega'$.  Hence, it suffices to produce a vector $v \in H_1(\Omega, \Sigma', \bC)$ that is not contained in $H_1(\Omega', \Sigma'', \bC)$ that is also linearly independent from every vector in $H_1(\Omega', \Sigma'', \bC)$.  Since $\theta \not= \theta'$ and $\Omega' \subset \Omega$, the boundary of $\Omega$ is not contained in $\Omega'$.  The boundary of $\Omega$ consists of a union of parallel saddle connections because this is true of the boundary of any invariant component.

Let $\Omega'^c = (\hat X, \hat \omega) \setminus \Omega'$, and let $\Sigma''^c = \Sigma \setminus \Sigma''$.  Let $v$ be the vector given by considering the sum of all of the boundary saddle connections of $\Omega$ corresponding to the elements of $H_1(\Omega, \Sigma', \bC)$.  Without loss of generality, let the boundary of $\Omega'$ lie in the vertical foliation.  Then by definition of $v$, $v \in H_1(\Omega, \Sigma', \bC)$ and by the observation above, $v \in H_1(\Omega'^c, \Sigma''^c, \bC)$.  If we deform $\Omega'^c$, while fixing $\Omega'$, by multiplying the period vectors contained in $\Omega'^c$ by the upper triangular matrices of $\text{GL}_2(\bR)$ in a neighborhood of the identity of $\text{GL}_2(\bR)$ that fixes the vertical foliation, then we see that $v$ can be varied in a $1$-complex dimensional space that is independent of any vector in $H_1(\Omega', \Sigma'', \bC)$.  Hence, the observation preceding this lemma implies that $v$ represents a vector in $H_1(\Omega, \Sigma', \bC)$ that is linearly independent from the vectors of $H_1(\Omega', \Sigma'', \bC)$.  Thus, $d(\Omega') < d(\Omega)$.
\end{proof}

\begin{definition}
Let $\{\Omega(\theta_n)\}$ be a sequence of invariant components contained in a slit translation surface $(\hat X, \hat \omega)$ such that the sequence $\{\theta_n\}$ converges to $\theta'$.  Define the \emph{$\omega$-limit set} of a sequence of invariant components to be
$$\Omega(\{\Omega(\theta_n)\}) := \bigcap_{n=1}^{\infty} \overline{\bigcup_{k \geq n}\Omega(\theta_k)}.$$
We say that $\{\Omega(\theta_n)\}$ \emph{accumulates to} $\Omega(\{\Omega(\theta_n)\})$.
\end{definition}

The following lemma can be proven with the exact same argument as Lemma \ref{InfSCImpFinCyl}.

\begin{lemma}
\label{GenOmLimSet}
Given a slit translation surface $(\hat X, \hat \omega)$ containing an infinite sequence of invariant components $\{\Omega(\theta_n)\}$, the set $\Omega(\{\Omega(\theta_n)\})$ is a union of parallel invariant components.
\end{lemma}

We would like to proceed by induction on the dimension of invariant components.  However, we are missing a key ingredient.  While $\Omega(\{\Omega(\theta_n)\})$ is a union of invariant components, it is not at all obvious that any individual invariant component in the sequence is contained in the $\omega$-limit set $\Omega(\{\Omega(\theta_n)\})$.  In fact, for general dynamical systems, this is false.  The following lemma resolves this issue.

\begin{lemma}
\label{SubseqContOmLimSet}
Given a slit translation surface $(\hat X, \hat \omega)$ containing an infinite sequence of invariant components $\{\Omega(\theta_n)\}$ each of which is a proper subset of $(\hat X, \hat \omega)$, there exists a subsequence $\{\Omega(\theta_{n_k})\}$ such that for all $k$, $\Omega(\theta_{n_k}) \subsetneq \Omega(\{\Omega(\theta_n)\})$.
\end{lemma}

\begin{proof}
By contradiction, assume that there exists $N$ such that for all $n \geq N$, $\Omega(\theta_n) \not\subset \Omega(\{\Omega(\theta_n)\})$.  We pass to a subsequence of $\{\Omega(\theta_n)\}$ with the property that for all $k$, there exists a trajectory in each of the invariant components in the sequence $\{\Omega(\theta_{n_k})\}$ that intersects a saddle connection $\sigma$ in the boundary of $\Omega(\{\Omega(\theta_n)\})$.  From here we follow the proof of Lemma \ref{InfSCImpFinCyl} and use the trajectories emanating from the saddle connection $\sigma$ to produce either an invariant component not contained in $\Omega(\{\Omega(\theta_n)\})$, a contradiction, or a closed union of saddle connections bordering a cylinder that was not assumed to be contained in $\Omega(\{\Omega(\theta_n)\})$, yet the argument from Lemma \ref{InfSCImpFinCyl} proves that it is, again yielding a contradiction.  This contradiction proves that there is no such $N$ assumed in the contradiction assumption above.  Therefore, the sequence $\{\Omega(\theta_n)\}$ admits a subsequence $\{\Omega(\theta_{n_k})\}$ of invariant components such that for all $k$, $\Omega(\theta_{n_k}) \subseteq \Omega(\{\Omega(\theta_n)\})$.

Since $\Omega(\theta_{n_k})$ is a proper subset of the surface, if $\Omega(\{\Omega(\theta_n)\}) = (\hat X, \hat \omega)$, then we conclude.  Otherwise, $\Omega(\{\Omega(\theta_n)\})$ has boundary.  By abuse of notation, pass to a subsequence if necessary such that $\{\theta_{n_k}\}$ converges to $\theta'$.  Since $\theta_{n_k} \not= \theta'$ and no trajectory of angle $\theta_{n_k}$ can pass through the boundary of $\Omega(\{\Omega(\theta_n)\})$ for all $k$, $\Omega(\theta_{n_k})$ must indeed be a proper subset of $\Omega(\{\Omega(\theta_n)\})$ for all $k$.
\end{proof}

The following corollary follows from Lemmas \ref{LemInvCompImpSmallDim}, \ref{GenOmLimSet}, and \ref{SubseqContOmLimSet}, and the definition of an $\omega$-limit set of a sequence of invariant components.

\begin{corollary}
\label{CplxInc}
If $\{\Omega(\theta_n)\}$ accumulates to $\Omega$, then there exists a subsequence $\{\Omega(\theta_{n_k})\}$ such that for all $k$, $d(\Omega(\theta_{n_k})) < d(\Omega)$.
\end{corollary}

Let $(\hat X, \hat \omega)$ be a slit translation surface without boundary.  After forgetting the slits, but not the endpoints of the slits, we get a translation surface $(X, \omega)$ without boundary in a stratum $\cH(\kappa)$.  We say that $\cH(\kappa)$ is the \emph{stratum associated to $(\hat X, \hat \omega)$}.  Clearly, $d(\hat X, \hat \omega) = \dim_{\bC}\cH(\kappa)$.

\begin{theorem}
\label{CBRankBd}
Let $(\hat X, \hat \omega)$ be a slit translation surface associated to the stratum $\cH(\kappa)$.  Let $r$ denote the Cantor-Bendixson rank of $\Theta(\hat X, \hat \omega)$.  If $d = \dim_{\bC} \cH(\kappa)$, then
$$1 \leq r \leq d.$$
\end{theorem}

\begin{proof}
We proceed by induction on the dimension of the invariant components contained in $(\hat X, \hat \omega)$.  If $\Theta(\hat X, \hat \omega)$ is finite, then we are done.  Consider a sequence of saddle connections associated to the points in $\Theta(\hat X, \hat \omega)$.  By Lemma \ref{InfSCImpFinCyl}, any sequence of saddle connections contains a subsequence converging to a union of invariant components.  All invariant components have dimension at least two (where dimension two is realized by a cylinder with one marked point in each of its boundaries).  Hence, every element in $\Theta(\hat X, \hat \omega)^*$ represents the direction of an invariant component with dimension at least two.  By the induction hypothesis, assume that every element of $\Theta(\hat X, \hat \omega)^{*(k-1)}$ represents a direction containing an invariant component with dimension at least $k$.  If there is a sequence of points in $\Theta(\hat X, \hat \omega)^{*(k-1)}$, then they correspond to a sequence of invariant components with dimension at least $k$ by assumption and they accumulate to a union of parallel invariant components $\Omega$ by Lemma \ref{GenOmLimSet}.  Then $\Omega$ has dimension at least $k+1$ by Corollary \ref{CplxInc}.  Since the dimension of $(\hat X, \hat \omega)$ is bounded above by the dimension of $\cH(\kappa)$, this process must terminate after at least $d$ iterations in which case the Cantor-Bendixson rank is bounded above by $d$.
\end{proof}

In light of the discussion in Section \ref{AppsSurgCBRkSect}, we see that Theorem \ref{FinCBRankMainThm} follows from Theorem \ref{CBRankBd}.

\section{The Examples $\hat S_n$ Revisited}

This allows us to revisit the family of slit translation surfaces $\hat S_n$ in the previous section and conclude

\begin{corollary}
\label{CBRankFamExact}
The Cantor-Bendixson rank of $\Theta(\hat S_n)$ is $n+2$.
\end{corollary}

This proves that the upper bound in Theorem \ref{CBRankBd} is in fact tight.  We leave it to the reader to produce examples from each $\hat S_n$ by adding only vertical and horizontal slits on the boundaries of the rectangular depictions of the cylinders in Figure \ref{FamCBRkn} so that every intermediate Cantor-Bendixson rank is realized in each stratum containing $\hat S_n$.  We conclude with an open problem.

\

\noindent \textbf{Open Problem.} Can the maximum Cantor-Bendixson rank established in Theorem \ref{CBRankBd} be realized by a slit translation surface associated to every connected component of a stratum of Abelian differentials?

\bibliography{fullbibliotex}{}

\providecommand{\bysame}{\leavevmode\hbox to3em{\hrulefill}\thinspace}
\providecommand{\MR}{\relax\ifhmode\unskip\space\fi MR }
\providecommand{\MRhref}[2]{%
  \href{http://www.ams.org/mathscinet-getitem?mr=#1}{#2}
}
\providecommand{\href}[2]{#2}
\begin{thebibliography}{GMN13}

\bibitem[AEZ14]{AthreyaEskinZorichCountGenJSDiffs}
Jayadev~S. Athreya, Alex Eskin, and Anton Zorich, \emph{Counting generalized
  {J}enkins-{S}trebel differentials}, Geom. Dedicata \textbf{170} (2014),
  195--217. \MR{3199485}

\bibitem[Aul15]{AulicinoCompDegKZ}
David Aulicino, \emph{Teichm\"uller discs with completely degenerate
  {K}ontsevich--{Z}orich spectrum}, Comment. Math. Helv. \textbf{90} (2015),
  no.~3, 573--643. \MR{3420464}

\bibitem[Boi15a]{BoissyConnCompsStrMeroDiffs}
Corentin Boissy, \emph{Connected components of the strata of the moduli space
  of meromorphic differentials}, Comment. Math. Helv. \textbf{90} (2015),
  no.~2, 255--286. \MR{3351745}

\bibitem[Boi15b]{BoissyModSpMeroDiffs}
\bysame, \emph{Moduli space of meromorphic differentials with marked horizontal
  separatrices}, Preprint \textbf{arXiv:1507.00555} (2015), 1--33.

\bibitem[BS15]{BridgelandSmithStabConds}
Tom Bridgeland and Ivan Smith, \emph{Quadratic differentials as stability
  conditions}, Publ. Math. Inst. Hautes \'Etudes Sci. \textbf{121} (2015),
  155--278. \MR{3349833}

\bibitem[Fen15]{FenyesAbelSL2RLocSys}
Aaron Fenyes, \emph{Abelianization of $\text{SL}(2,\mathbb{R})$ local systems},
  Preprint \textbf{arXiv:1510.05757} (2015), 1--81.

\bibitem[GMN13]{GMNWallHitchinWKB}
Davide Gaiotto, Gregory~W. Moore, and Andrew Neitzke, \emph{Wall-crossing,
  {H}itchin systems, and the {WKB} approximation}, Adv. Math. \textbf{234}
  (2013), 239--403. \MR{3003931}

\bibitem[Gup14]{GuptaMeroQuadHalfPlane}
Subhojoy Gupta, \emph{Meromorphic quadratic differentials with half-plane
  structures}, Ann. Acad. Sci. Fenn. Math. \textbf{39} (2014), no.~1, 305--347.
  \MR{3186818}

\bibitem[GW15]{GuptaWolfQuadDiffsHalfPlaneStrcts}
Subhojoy Gupta and Michael Wolf, \emph{Quadratic differentials, half-plane
  structures, and harmonic maps to graphs}, Preprint \textbf{arXiv:1505.02939}
  (2015), 1--39.

\bibitem[HM79]{HubbardMasur}
John Hubbard and Howard Masur, \emph{Quadratic differentials and foliations},
  Acta Math. \textbf{142} (1979), no.~3-4, 221--274. \MR{523212 (80h:30047)}

\bibitem[Kec95]{KechrisClassDescSetTheory}
Alexander~S. Kechris, \emph{Classical descriptive set theory}, Graduate Texts
  in Mathematics, vol. 156, Springer-Verlag, New York, 1995. \MR{1321597}

\bibitem[Mas86]{MasurClosedTraj}
Howard Masur, \emph{Closed trajectories for quadratic differentials with an
  application to billiards}, Duke Math. J. \textbf{53} (1986), no.~2, 307--314.
  \MR{MR850537 (87j:30107)}

\bibitem[Str84]{Strebel}
Kurt Strebel, \emph{Quadratic differentials}, Ergebnisse der Mathematik und
  ihrer Grenzgebiete (3) [Results in Mathematics and Related Areas (3)],
  vol.~5, Springer-Verlag, Berlin, 1984. \MR{MR743423 (86a:30072)}

\bibitem[Via08]{VianaIETBook}
Marcelo Viana, \emph{Dynamics of interval exchange transformations and
  {T}eichm\"uller flows}, 2008.

\end{thebibliography}

\end{document}